\documentclass[onefignum,onetabnum]{siamart190516}

\usepackage{amsfonts}
\usepackage{graphicx}
\usepackage{algorithmic}
\usepackage{amsmath, amssymb}
\usepackage{mymacros}
\providecommand{\introduce}[1]{\textit{#1}}

\newsiamremark{remark}{Remark}
\newsiamremark{hypothesis}{Hypothesis}
\newsiamremark{ex}{Example}
\crefname{hypothesis}{Hypothesis}{Hypotheses}

\headers{Opinion Dynamics on Discourse Sheaves}{J. Hansen and R. Ghrist}

\title{Opinion Dynamics on Discourse Sheaves\thanks{Submitted to the editors DATE.
\funding{This work was funded by the Office of the Assistant Secretary of Defense Research \&
	Engineering through a Vannevar Bush Faculty Fellowship, ONR N00014-16-1-2010.}}}

\author{Jakob Hansen\thanks{Department of Mathematics, University of
    Pennsylvania, Philadelphia, PA} 
  (\email{jhansen@math.upenn.edu})
\and Robert Ghrist\thanks{Department of Mathematics and Department of Electrical
  and Systems Engineering, University of Pennsylvania, Philadelphia, PA 
  (\email{ghrist@math.upenn.edu}).}}

\usepackage{amsopn}

 \ifpdf
 \hypersetup{
   pdftitle={Opinion Dynamics on Discourse Sheaves},
   pdfauthor={J. Hansen and R. Ghrist}
 }
 \fi

\begin{document}

\maketitle

% REQUIRED
\begin{abstract}
  We introduce a novel class of Laplacians and diffusion dynamics on {\em discourse sheaves} as a
  model for network dynamics, with application to opinion 
  dynamics on social networks. These {\em sheaves} are algebraic data structures tethered to 
  a network (or more general space) that can represent various modes of communication, 
  including selective opinion modulation and lying. After introducing the sheaf model, 
  we develop a sheaf Laplacian in this context and show how to evolve both opinions
  and communications with diffusion dynamics over the network.
  Issues of controllability, reachability, bounded confidence, and harmonic extension
  are addressed using this framework.   
\end{abstract}

% REQUIRED

% REQUIRED

\section{Introduction}
Social networks are one of the principal motivating examples for the study of complex networks. 
Among the many interesting problems associated with social networks, {\em opinion dynamics} --- the study  
of how preferences or opinions emerge and evolve --- are especially interesting, blending ideas 
from dynamical systems and graph theory. 
Structural effects of the network on opinion dynamics began with the analysis of linear dynamical models
\cite{taylor_towards_1968,IMGroot,friedkin_social_1990} and have developed into more sophisticated formulations \cite{deffuant_mixing_2000,dittmer_consensus_2001,hegselmann_opinion_2002}, including features such as bounded confidence.
This paper introduces both a novel model and a novel set of tools for the analysis of opinion dynamics. 
After a brief review of classical opinion dynamics models, we survey the results of this paper. 

% ----------------------------------------------------------------------------------
\subsection{Models of Opinion Dynamics}
\label{sec:Models}
% ----------------------------------------------------------------------------------

Early models of opinion dynamics used linear network dynamics to evolve
single-dimensional preferences. Consider a social network represented as an
undirected graph $G=(V,E)$ of vertices and edges. The state space for
single-opinion real dynamics is $\R^{V}$, with $x \in \R^V$ representing a
distribution of preferences $x_v\in\R$ at each vertex, ranging from positive to
indifferent (null) to negative. Graph-based linear dynamics evolve preferences
over time.

In continuous-time models, the graph Laplacian, $L$, generates dynamics via the graph diffusion equation
\begin{equation}
\label{eqn:continuoustimeopinions}
  \frac{dx}{dt} = -\alpha L x, \quad : \quad \alpha>0 ,
\end{equation}
perhaps with slight modifications %like a constant term
\cite{abelson_mathematical_1964,taylor_towards_1968}. Analogous
models~\cite{french_formal_1956,IMGroot,lehrer_when_1976} were studied
in discrete time, with a (typically stochastic) state evolution matrix following the sparsity pattern
of the adjacency matrix, $A$, of the social network:
\begin{equation}\label{eqn:discretetimeopinions}
  x[t+1] = A x[t].
\end{equation}
Variations with more terms added additional richness to the models~\cite{friedkin_social_1990}.

Without modification, these linear graph models result in asymptotically stable equilibria at
a single consensus opinion over the network. While this global fixed consensus may be useful in some 
situations (e.g., flocking in robotics/swarm applications \cite{jadbabaie_coordination_2003,tanner_stable_2003}), 
it does not represent the typical behavior of opinion distributions in social
networks. Indeed, a central problem in the study of opinion dynamics is the
construction of simple models that replicate one of the most salient features of
real-world opinions: the existence of polarization or failure to come to
consensus, known as the \introduce{community cleavage problem}
\cite{friedkin_problem_2015}. The earliest approaches to solving this problem
added extra constant terms to dynamics of the
form~\eqref{eqn:continuoustimeopinions} or~\eqref{eqn:discretetimeopinions} that
encouraged diversity in opinions across the
network~\cite{taylor_towards_1968,friedkin_social_1990,friedkin_social_1999}.
These extra terms were interpreted as external influences or the effect of
individual stubbornness on opinions. 

Much recent work has focused on a nonlinear extension of the discrete-time
dynamics that adds {\em bounded confidence} to agents' evaluation of their
neighbors' opinions \cite{deffuant_mixing_2000,dittmer_consensus_2001,hegselmann_opinion_2002}.
The most famous such model, popularized by Hegselmann and
Krause, posits a threshold $r$ for opinion sharing. If agents $i$ and $j$ have
opinions that differ by more than $r$, they do not communicate. Otherwise, they
influence each other's opinions linearly as in the discrete-time dynamics above.
That is,
\begin{equation}
	x_i[t+1] = \sum_{i \sim j} A_{ij} x_j \mathbf{1}_{\abs{x_i - x_j}<r}.
\end{equation}	
These dynamics are complex enough to produce bifurcated opinion distributions
without external influences, while admitting some direct analysis. 

Not all work on opinion dynamics has used the Laplacian or adjacency-based
linear network interaction formulations above. Other popular models include
dynamics with discrete opinion
spaces~\cite{sznajd-weron_opinion_2000,castellano_statistical_2009} based on
formalisms from statistical physics, as well as non-agent-based models that seek
to understand the overall distribution of opinions in a population without
tracking any individual's stance~\cite{weidlich_statistical_1971,canuto_eulerian_2012}.

In this paper, we will focus on network-based continuous-time models for opinion
dynamics with continuous opinion spaces; discrete-time examples can typically be
extracted from these via Euler discretization. Our models will find most kinship
with those that approach opinion dynamics from a control-theoretic or systems
analysis perspective, like those discussed
in~\cite{proskurnikov_tutorial_2017,proskurnikov_tutorial_2018}.

% ----------------------------------------------------------------------------------
\subsection{Contributions of This Work}
\label{sec:Contributions}
% ----------------------------------------------------------------------------------

In this paper, we introduce a novel approach to networked opinion dynamics, using 
ideas from {\em sheaf theory}. This subject, commonly used in algebraic topology and algebraic 
geometry, is vast
\cite{kashiwara_sheaves_1990,dimca_sheaves_2004,gelfand_methods_2003,hartshorne_algebraic_1977},
but has a simple, graph-based reduction that amounts
to little more than a networked system of linear transformations (see \S\ref{sec:CellularSheaves}
for definitions). There are three key ideas from sheaf theory that we use:
\begin{enumerate}
	\item 
	{\em Cellular sheaves} are a topological data structure for graphs (or more general cell complexes) 
	\cite{curry_sheaves_2014}, which, we argue, permits modeling of very sophisticated opinion state spaces and 
	communication strategies.
	\item 
	{\em Sheaf cohomology} is an algebraic-topological invariant of the sheaves over graphs. 
	We demonstrate its use in computing obstructions to solving problems of opinion dynamics on graphs.  
	\item 
	{\em Sheaf Laplacians} are far-reaching extensions of the graph Laplacian \cite{hansen_toward_2019}
	and extend ideas of harmonic flow to a wide array of opinion dynamics models.  
\end{enumerate}
Although not a familiar toolset within network science, sheaf-theoretic methods have a remarkable range of
expressiveness and can encode more realistic modes of expression of opinions, including exaggeration, modulation, 
and selective obfuscation. 

After setting up the relevant mathematical structures in \S\ref{sec:CellularSheaves}, we proceed
directly to the contributions, summarized below.
\begin{enumerate}
	\item In \S\ref{sec:DiscourseSheaves}, we introduce our model of opinions over a social network via a 
	cellular sheaf. In this model, each agent has an {\em opinion space} (a vector space with dedicated basis) and
	each edge has an independent {\em discourse space} (a vector space with basis of topics up for discussion). 
	Expression of opinions is programmed via linear transformations from opinion to discourse spaces, allowing 
	for private opinions selectively expressed or combined into policies. 
	\item In \S\ref{sec:SheafDiffusion}, we use the sheaf Laplacian to set up diffusion dynamics on the 
	discourse sheaf, proving asymptotic convergence of initial opinions to a (literal) {\em harmonic} state: all
	agents express opinions in harmony with neighbors, though the private opinions of neighbors may be distinct (or even incomparable). 
	\item In \S\ref{sec:HarmonicExtension}, we use sheaf cohomology to characterize whether certain problems
	of extension and convergence have solutions. For example, if certain agents are inflexible and  
	will not modify their opinions, does there exist a unique global solution through modification of others' opinions? 
	We show that this is a problem of {\em harmonic extension}, determinable via a cohomology computation.
	\item The question of manipulation of a system through inflexible agents leads naturally to questions of 
	a control-theoretic nature. In \S\ref{sec:Control}, sheaf cohomology is shown
  to determine the 
	controllability and observability of opinions. 
	\item Using the Laplacian to evolve individual opinions in order to come to harmonic expression is only half 
	the picture. One could instead keep opinions fixed and evolve the expression of opinions in order to reduce
	discord. In \S\ref{sec:EvolvingExpression}, we extend the Laplacian diffusion model to the sheaf maps that
	express opinions. This leads to the interesting phenomenon of agents ``learning to lie'' to reach concord.
	The natural extension to joint opinion-expression diffusion is given in \S\ref{sec:Joint}.
	\item Finally, in \S\ref{sec:NonlinearLaplacians}-\ref{sec:AntagonisticDynamics}, we move from linear 
	to (slightly) nonlinear dynamics of opinion distributions on 
	discourse sheaves, showing how to mimic the {\em bounded confidence} models of
  \cite{hegselmann_opinion_2002} and the antagonistic social dynamics of \cite{altafini_predictable_2015}.
\end{enumerate}

% ============================================================================================
\section{Introduction to Cellular Sheaves}
\label{sec:CellularSheaves}
% ============================================================================================
For purposes of this paper, a cellular sheaf is a data structure augmenting a graph that describes
consistency relationships for algebraic data attached to the graph. For simplicity, we work with
vector spaces and linear transformations thereof. At one or two points, it will be convenient if
not necessary to have all vector spaces real, finite-dimensional, and with an inner product structure. 
For the remainder of this work, we will implicitly assume these conditions: see
\cite[Section 3.3]{hansen_toward_2019}
for details on how to deal with more general Hilbert spaces. Further, for
simplicity, we will assume that every vector space has a canonical orthonormal
basis identifying it with $\R^n$ for some $n$, and identify linear maps with
their representing matrices and the inner product $\ip{x,y}$ with the standard
inner product $x^Ty$ on $\R^n$.

% ----------------------------------------------------------------------------
\subsection{Definitions}
\label{sec:Defs}
% ----------------------------------------------------------------------------
\begin{def}
\label{def:sheaf}
  Let $G$ be a graph. A cellular sheaf $\Fc$ on $G$ is specified by the
  following data:
  \begin{itemize}
  \item
    a vector space $\Fc(v)$ for each vertex $v$ of $G$,
  \item
     a vector space $\Fc(e)$ for each edge $e$ of $G$, and
  \item
    a linear map $\Fc_{v \face e} : \Fc(v) \to \Fc(e)$ for each incident
    vertex-edge pair $v \face e$.
  \end{itemize}
\end{def}
The vector spaces $\Fc(v)$ are called the \introduce{stalks} over $v$, and the linear
maps $\Fc_{v \face e}$ are the \introduce{restriction maps}. The terminology seems unmotivated
and is inherited from more general sheaf theory \cite{curry_sheaves_2014}. Thinking in terms of data (stalks) 
and communication (restriction maps) is perhaps preferable in the context of this paper, 
see Figure~\ref{fig:sheafcartoon}.

% ********************************************************************************************
\begin{figure}
\label{fig:sheafcartoon}
	\centering
	\includegraphics[width=5.0in]{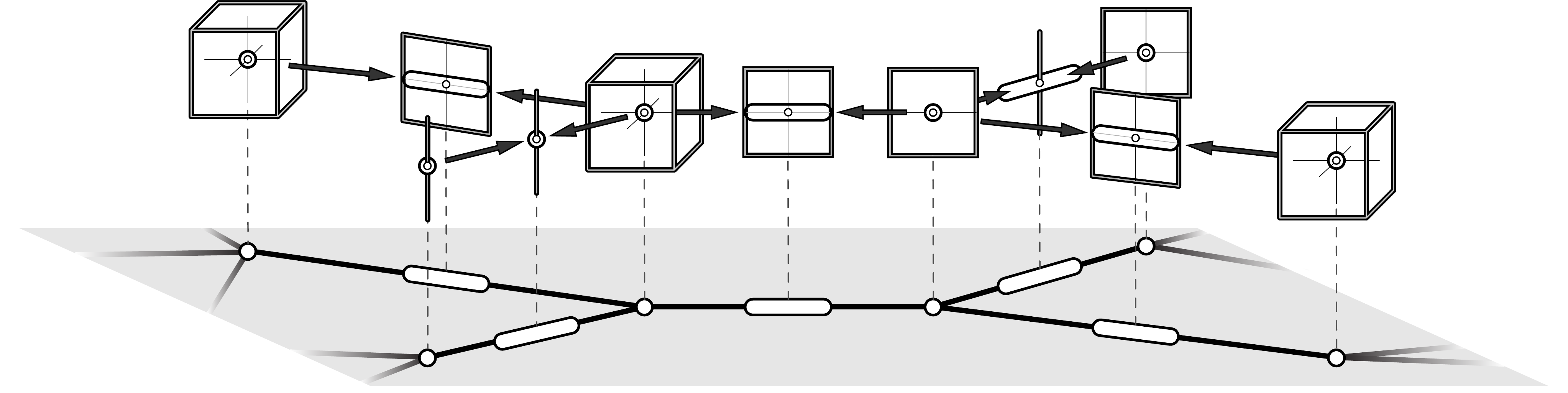}
	\caption{A cartoon diagram of a cellular sheaf over a graph. Vector spaces of varying dimensions
	over vertices and edges are attached via linear transformations, following the adjacency
	pattern of the graph. The entire system of linear transformations forms the cellular sheaf.}
\end{figure}
% ********************************************************************************************

The simplest example of a cellular sheaf is the \introduce{constant sheaf}.
For $V$ a fixed vector space, the constant $V$-sheaf on $G$, denoted $\underline V$, attaches
a copy of $V$ to each vertex and edge (all stalks are $V$), with
$\underline{V}_{v\face e} = \id_V$ for all incident vertex-edge pairs (all restriction maps
are the identity). One interprets a constant sheaf as specifying that all vertices access
the same data (from $V$) and communicate perfectly with their neighbors over edges. 

There is more to a sheaf than its stalks: replacing all the identity maps in $\underline V$ 
with zero-maps ($\Fc_{v\face e}:V\mapsto{0}$) gives a very different sheaf in which all
communication is devoid of information. Such a sheaf is a jumble in contrast to the tightly 
coordinated constant sheaf $\underline{V}$.

% ----------------------------------------------------------------------------
\subsection{Sections and Sheaf Cohomology}
\label{sec:Cohomology}
% ----------------------------------------------------------------------------
If one visualizes a sheaf of vector spaces over a graph as being a network of vector 
spaces and linear transformations, then one is naturally led to questions of how to 
generalize the familiar notions of linear algebra --- kernels, images, etc. --- to such
a networked structure. This is the impetus for {\em homological algebra} and the 
{\em cohomology} of a sheaf.

One begins by bundling all the data over vertices and over edges into a pair of 
conglomerated vector spaces. These are called spaces of \introduce{cochains}
\begin{align}
  C^0(G;\Fc) &= \bigoplus_{v \in V(G)} \Fc(v) \\
  C^1(G;\Fc) &= \bigoplus_{e \in E(G)} \Fc(e).
\end{align}
Elements of $C^0$ are called \introduce{0-cochains}: such an
$x \in C^0(G;\Fc)$ consists of a choice of data, $x_v \in \Fc(v)$, for every vertex $v$ of $G$. 
Likewise, an element $y\in C^1(G;\Fc)$ is called a \introduce{1-cochain} and is a choice of data indexed
over the edges of $G$. 

In the same manner that edges and vertices are stitched together to form a graph, 
data over the vertices (0-cochains) and edges (1-cochains) are tied together 
via a linear transformation --- the \introduce{coboundary} map, $\delta \colon C^0(G;\Fc) \to C^1(G;\Fc)$.
To define $\delta$ explicitly, choose a fixed but arbitrary orientation on each edge $e$. Then
the evaluation of $\delta$ on an oriented edge $e = u\to v$ is defined as follows:  
\begin{equation}
(\delta x)_e = \Fc_{v \face e} x_v - \Fc_{u \face e} x_u .
\end{equation}
The orientation merely serves as a choice of basis elements for defining the difference
operation in $\delta$. That choice is irrelevant, as what one cares about is the
kernel of $\delta$. 

If one thinks about the coboundary $\delta$ as a measure of ``disagreement'' of data
across an edge, then $\ker\delta$ is the subspace of $C^0(G;\Fc)$ consisting of choices
of data over the vertex set which ``agree'' over the edges. That is, for
every $e = u \sim v$, $\Fc_{u \face e} x_u = \Fc_{v \face e} x_v$. The set of all 
such solutions to the global constraint satisfaction problem of the sheaf has the
structure of a vector subspace of $C^0(G;\Fc)$ and goes by the title of {\em cohomology}.

\begin{def}
	\label{def:cohomology}
	For $\Fc$ a cellular sheaf on a graph $G$, the \introduce{zeroth cohomology} $H^0(G;\Fc)$, 
	also known as the space of \introduce{global sections} of $\Fc$, is
\begin{equation}
	H^0(G;\Fc) = \ker\delta \subset C^0(G;\Fc).
\end{equation}
\end{def}

The global sections of a sheaf are the global solutions to the networked
system of constraint equations programmed into the restriction maps $\Fc_{u\face e}$ 
of the data structure over the graph. The cohomological terminology comes from the
algebraic topology of sheaves \cite{kashiwara_sheaves_1990,dimca_sheaves_2004,gelfand_methods_2003} 
--- a beautiful theory that can be ignored by the end-user of the models in this paper 
(but which secretly animates many of the results).

Related to the space of global sections are the subspaces of \introduce{local sections}. 
For $A$ a subgraph of $G$, we let $C^0(A;\Fc) =
\bigoplus_{v \in V(A)} \Fc(v)$ and $C^1(A;\Fc) = \bigoplus_{e \in E(A)} \Fc(e)$.
The coboundary $\delta$ restricts to a map $C^0(A;\Fc) \to C^1(A;\Fc)$; its
kernel is the space of local sections over $A$, denoted
$H^0(A;\Fc)$. This is the subspace of $C^0(A;\Fc)$ which is consistent over
every edge in $A$.

Dual to the local sections over $A$ is the cohomology relative to $A$. We let
$C^0(G,A;\Fc) = \bigoplus_{v \notin V(A)} \Fc(v)$ and $C^1(G,A;\Fc) =
\bigoplus_{e \notin E(A)} \Fc(e)$, and restrict $\delta$ to a map between these
spaces. The \introduce{degree 0 relative cohomology} is $H^0(G,A;\Fc) = \ker \delta|_{C^0(G,A;\Fc)}$. Local
sections $H^0(A;\Fc)$ should be thought of as assignments to a subset of
vertices that are consistent on $A$, while the relative cohomology $H^0(G,A;\Fc)$
can be viewed as global sections of $\Fc$ on $G$ that vanish on $A$.
\begin{ex}
  For the constant sheaf $\underline \R$ on a connected graph $G$,
  $H^0(G;\underline \R)$ is a one-dimensional vector space spanned 
  by the constant functions on vertices of $G$. For any nonempty subgraph $A$
  of $G$, $H^0(A;\underline \R)$ consists of functions on the vertices of $A$
  which are locally constant on the subgraph; its dimension is the number of
  connected components of $A$. The relative cohomology $H^0(G,A;\underline
  \R)$ is zero-dimensional, since any constant function which is zero on $A$ 
  must be zero everywhere on $G$. 
\end{ex}

% ----------------------------------------------------------------------------
\subsection{The Sheaf Laplacian}
\label{sec:SheafLaplacian}
% ----------------------------------------------------------------------------
Recall that for a graph with signed incidence matrix $B$, the 
\introduce{graph Laplacian} is given by $L=BB^T$. 
These are combinatorial versions of the familiar second-order differential
operator with enormous applicability across combinatorics, data science, and more 
\cite{chung_spectral_1992,belkin_laplacian_2003,coifman_diffusion_2006}. 
Less well-known in applied mathematics is the definition of Laplacians of complexes of 
sheaves of inner-product spaces over topological spaces: these arise in Hodge theory 
and algebraic geometry \cite{nicolaescu_lectures_2007}.
Between the graph Laplacian and Hodge Laplacian lies a mean notion of a Laplacian for
cellular sheaves on graphs \cite{hansen_toward_2019}. 

The construction is uncomplicated. 
Observe that the coboundary $\delta$ of the constant sheaf $\underline{\R}$ 
with stalk $\R$ equals the transposed signed incidence matrix $B^T$ 
of the graph $G$. For a sheaf $\Fc$ over $G$, $\delta$ may be seen as a 
generalized incidence matrix for $\Fc$. The potential variation in stalk
dimensions means that $\delta$ is a block matrix, and the restriction maps
determine the block entries, with sparsity pattern determined by the structure of $G$.
\begin{def}
	\label{def:Laplacian}
	For a sheaf $\Fc$ on a graph $G$, the \introduce{sheaf Laplacian} is 
\begin{equation}
	L_\Fc = \delta^T \delta \colon C^0(G;\Fc)\to C^0(G;\Fc). 
\end{equation}
\end{def}
Just as the graph Laplacian does not depend on the orientations chosen for the
edges in the signed incidence matrix, the sheaf Laplacian does not depend on the
choice of orientations for the construction of the coboundary $\delta$.

The following theorem may seem trivial in the context of sheaves over graphs, but it 
stems from deeper results (on sheaves over higher-dimensional
cell complexes and with relations to higher cohomologies).  

\begin{theorem}[Hodge Theorem]
For $\Fc$ a sheaf on a graph $G$ as above, 
\begin{equation}
	H^0(G;\Fc) = \ker L_\Fc .
\end{equation}
\end{theorem}

\begin{ex} 
The well-known fact that the kernel of the graph Laplacian is the space of locally
constant functions on $G$ follows from applying the Hodge theorem to the constant sheaf on $G$.
\end{ex}

A straightforward computation shows that for a 0-cochain $x\in C^0(G;\Fc)$, 
the value of the sheaf Laplacian at a given vertex $v$ is
\begin{equation}
  (L_\Fc x)_v = \sum_{v, u \face e} \Fc_{v \face e}^T(\Fc_{v \face e} x_v -
  \Fc_{u \face e}x_u).
\end{equation}
This implies that the matrix of $L_\Fc$ has a block structure, with diagonal
blocks $L_{vv} = \sum_{v \face e} \Fc_{v\face e}^T\Fc_{v \face e}$ and
off-diagonal blocks $L_{vu} = - \Fc_{v \face e}^T\Fc_{u \face e}$. 

In the next section, we will interpret $(L_\Fc x)_v$ as a measure of the average disagreement 
of agent $v$ with its neighbors, or equivalently, of the disagreement of $v$ with an average
of its neighbors.

\begin{ex}
	The sheaf in Figure~\ref{fig:sheaflaplacian} has a coboundary map represented by the matrix
	\begin{equation}
    \delta = 
    \left[
      \begin{array}{c|cc|c|cc}
   -2 & 1 & -2 & 0  & 0  & 0  \\\hline
	0 & 1 & 1  & -1 & 0  & 0 \\\hline
	0 & 0 & 0  & -1 & -1 & 1 \\\hline
	0 & 0 & 0  & 0  &  0 & 1 \\
    1 & 0 & 0  & 0  & -1 & 0
      \end{array}
    \right] ,
	\end{equation}
	and therefore its Laplacian is
	\begin{equation}
	L_\Fc = \delta^T \delta =
  \left[
    \begin{array}{c|cc|c|cc}
      	5 & -2 & 4 & 0 & -1 & 0 \\\hline
	   -2 & 2 & -1 & -1 & 0 & 0 \\
	    4 & -1 & 5 & -1 & 0 & 0 \\\hline
	    0 & -1 & -1 & 2 & 1 & -1 \\\hline
	   -1 & 0 & 0 & 1 & 2 & -1 \\
        0 & 0 & 0 & -1 & -1 & 2
    \end{array}
      \right].
	\end{equation}
	The block columns of $\delta$ and $L_\Fc$ correspond to the stalks $\Fc(v_i)$ for $i=1,\ldots,4$.
\end{ex}

% ********************************************************************************************
\begin{figure}
\label{fig:sheaflaplacian}
	\centering
	\includegraphics[width=3.0in]{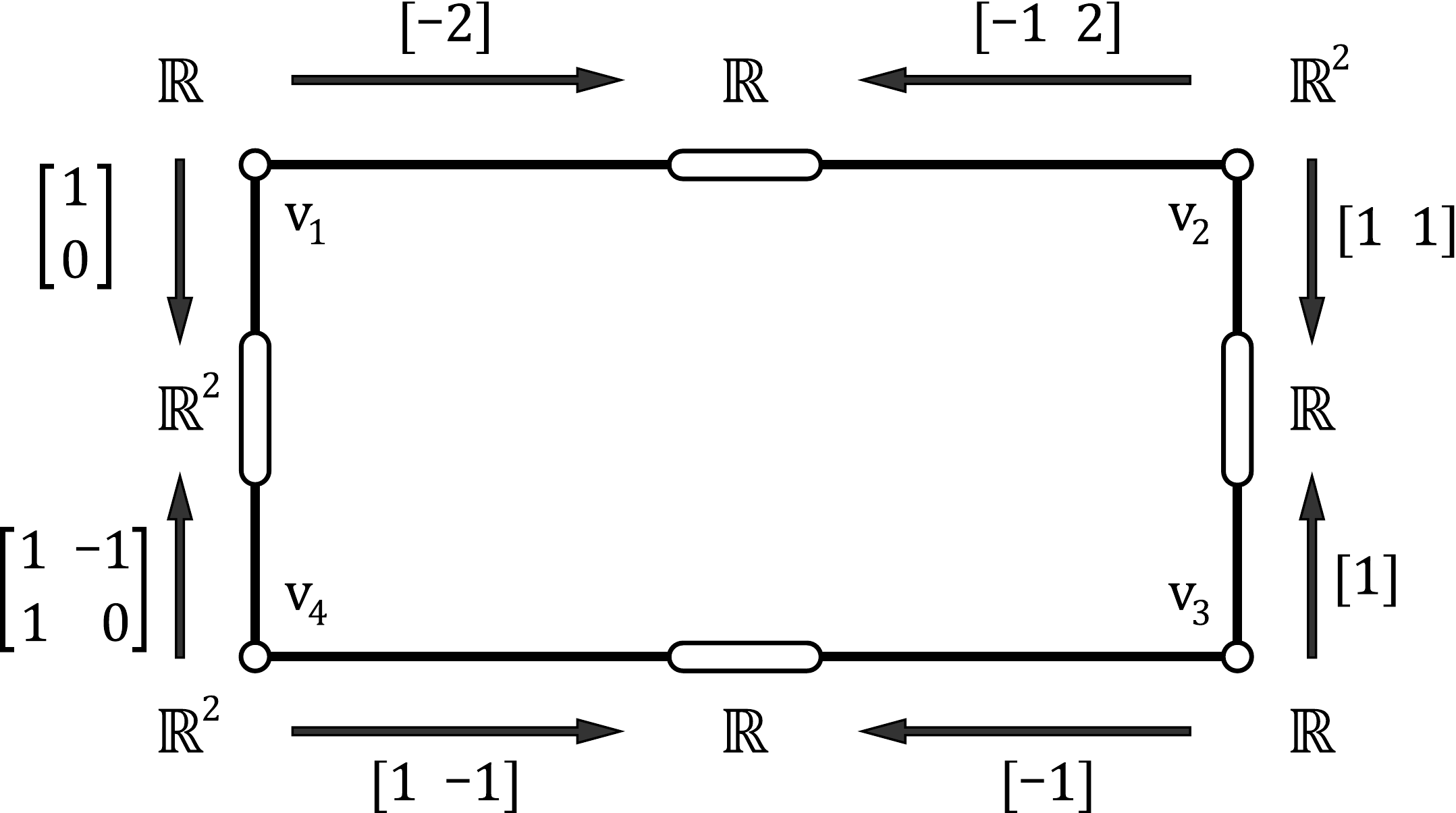}
	\caption{A cellular sheaf over a simple cyclic graph. All stalks are of dimension one or two.}
\end{figure}
% ********************************************************************************************

% ============================================================================================
\section{Discourse Sheaves}
\label{sec:DiscourseSheaves}
% ============================================================================================

We introduce a cellular sheaf model for opinions and discourse for which
Laplacian diffusion is an effective and computable method.

{\bf Construction:}
Given a social network presented as a graph $G$ with vertices representing agents and
(undirected) edges representing pairwise communication, consider the following
\introduce{discourse sheaf} $\Fc$. Each agent (vertex) $v$ has an \introduce{opinion space}, 
a real vector space with basis some collection of topics.  As with classical opinion
dynamics models, points on each axis correspond to negative, neutral, or positive opinions on
the topic, with a positive/negative intensity registered by the scalar value. This
opinion space comprises the stalk $\Fc(v)$ of the discourse sheaf over $v$. A choice
of element $x_v\in\Fc(v)$ is a vector recording the intensities of opinions or preferences 
on each of the \introduce{basis topics}. 

Given an edge $e$ between vertices $u$ and $v$, it is presumed that there is a certain set
of basis topics about which the two agent discuss. These are not necessarily the same as 
any of the basis topics from which $\Fc(u)$ or $\Fc(v)$ are generated; however, they do form the
basis of an abstract \introduce{discourse space}, $\Fc(e)$, the stalk over $e$.

Each agent represents their opinions on the topics of discussion by formulating
stances as a linear combination of existing opinions on personal basis topics. 
These expressions of opinion are linear transformations $\Fc_{u \face e}:\Fc(u)\to\Fc(e)$
and $\Fc_{v \face e}:\Fc(v)\to\Fc(e)$. If the agents hold opinions $x_u\in\Fc(u)$ and 
$x_v\in\Fc(v)$, then they have expressed consensus when $\Fc_{u\face e}(x_u)=\Fc_{v\face e}(x_v)$.
This is the local consistency condition implicit in the construction of a sheaf 
--- each edge imposes a linear consistency constraint on the stalks of its incident vertices.
Note that this does not imply that $u$ and $v$ have the same opinions: it means
that their expressions of personally held opinions have the appearance of agreement.

% ********************************************************************************************
\begin{figure}
\label{fig:discoursesheaf}
	\centering
	\includegraphics[width=4.75in]{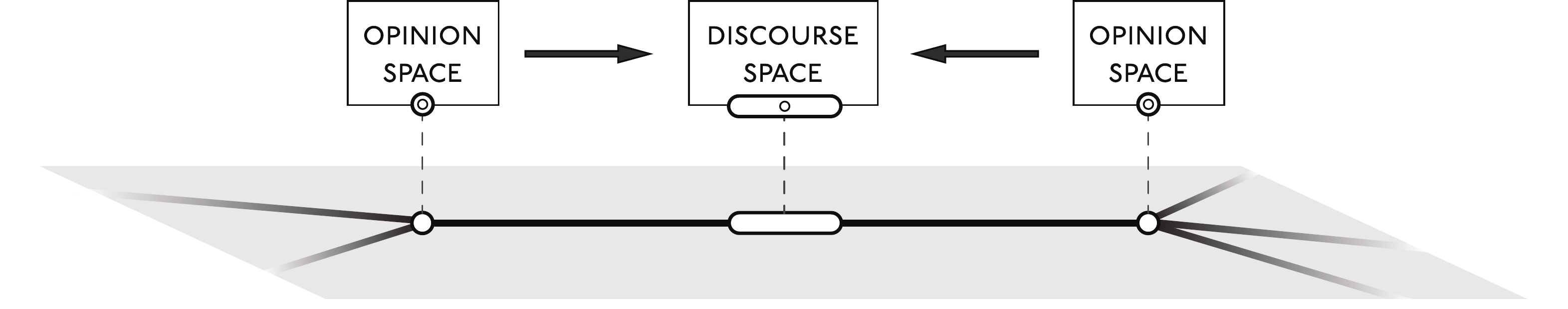}
	\caption{In a discourse sheaf, stalks over vertices are individual opinion spaces, stalks over edges 
	are discourse spaces, and restriction maps are expressions of opinions on the topics of discourse, 
	formulated linearly from basis opinions.}
\end{figure}
% ********************************************************************************************

The apparatus of \S\ref{sec:CellularSheaves} becomes clearer in the 
context of discourse sheaves. 
\begin{itemize}
	\item The constant sheaf $\underline{\R^n}$ is a discourse sheaf in which
    every agent
	has an opinion on the same $n$ basis topics, all of which are precisely expressed and discussed
	without embellishment. This is the implicit structure in most of the literature on 
	opinion dynamics. 
	\item A 0-cochain $x\in C^0(G;\Fc)$ is a private opinion distribution on the set of agents.
	\item A 1-cochain $\xi\in C^1(G;\Fc)$ is a distribution of expressed opinions over
	all pairwise agent discussions in the network. 
	\item The coboundary map $\delta:C^0\to C^1$ registers the ``difference of expression'' 
	between agents based on the expression of personal opinions. 
	\item The sheaf Laplacian $L_\Fc:C^0\to C^0$ registers the ``discord'' in the system. The 
	value of $L_\Fc(x)$ on each vertex $v$ represents the difference between $x_v$ (the opinions of 
	$v$) and the opinions which would bring $v$ in harmony with all its neighbors. 
	\item The zeroth cohomology $H^0(G;\Fc)$ computes the vector space of global 
	sections --- the space of opinion distributions in which all expressions
	of opinions are in harmony. As noted in Theorem \ref{thm:sheafdiffusionsoln}, these are literally
	{\em harmonic} (in the kernel of the Laplacian).
	\item Relative cohomology $H^0(G,A;\Fc)$ with respect to a subgraph $A\subset G$ represents
	harmonic opinion distributions which vanish on $A$. This is a measure of independence of 
	the agents in $A$ from the rest of the social network.  
\end{itemize}
The discourse sheaf model is, in one sense, a mild generalization of the usual consensus
problem over graphs. However, the ability to {\em program} a sheaf with  
linear transformations permits a number of features not present in the literature. 
Consider the simple example of three agents A, B, and C, all in pairwise
communication. 
\begin{itemize}
	\item Because stalks vary from vertex-to-vertex, the discourse sheaf permits agents to 
	have opinions on private basis topics. Agents A, B, and C need not have {\em any} basis topics in common.
	\item Because edge stalks are not identical to vertex stalks, the discourse sheaf model
	does not require everyone to share all their opinions with every neighbor; indeed, the
	topics for discussion need not relate at all to basis opinions of agents. Agents A and B 
	might be discussing whether to eat lunch at the nearby pub. The edge stalk (favorability of 
	dining at the pub) may not be something on which either A or B has a basis opinion. 
	\item The restriction maps allow for the formation of {\em policies} from {\em principles}.
	For example, if agent A has a strong basis-opinion preference for sandwiches and is neutral about noise, 
	the restriction map to the edge stalk could express a preference for dining at the (noisy, 
	sandwich-renowned) pub. Agent B, who has basis opinions about walking long distances (dislikes) 
	and quick meals (prefers) might have a restriction map that expresses dislike for the 
	(not nearby) pub.
	\item Positive scalar multiplication acts both on vectors (intensifying or damping opinions) and on 
	restriction maps (exaggerating or modulating expressed opinions). Negative scalar multiplication
	in a restriction map permits falsehoods: one can model agents who lie. Such dissembling or 
	deception can be done selectively. What C says to B need not match what C says to A 
	(even if they are discussing the same topic). 
	\item There are multiple ways to set up opinion dynamics on a discourse sheaf. We begin, following
	the classical literature, by having individual agents change their opinions over time. 
	This is perhaps not how real people engage in discourse. A different mode of evolution would
	permit {\em expression} of opinions to change, in order to bring discourse to a more 
	harmonious state. This is achievable in the sheaf model by setting up dynamics on the 
	restriction maps. Co-evolution of both opinion and expression is achievable in this model.
\end{itemize}

% =========================================================================================
\section{Sheaf Diffusion}
\label{sec:SheafDiffusion}
% =========================================================================================
Just as the graph Laplacian forms the basis for simple linear opinion dynamics,
so does the sheaf Laplacian on discourse sheaves. Consider the heat equation
\begin{equation}
\label{eqn:sheafheateqn}
  \frac{dx}{dt} = -\alpha L_\Fc x, \quad : \quad \alpha>0
\end{equation}
on $x \in C^0(G;\Fc)$. That is, $x$ represents an opinion distribution, where
$x_v \in \Fc(v)$ is the opinion of individual $v$. The diffusion dynamics tend
to push the value at a node $v$ toward greater agreement (as measured by the
sheaf structure) with the expressed opinions of its neighbors. Our first result
is that trajectories of this sheaf heat equation converge to global sections. 

\begin{theorem}
\label{thm:sheafdiffusionsoln}
  Solutions $x(t)$ to \eqref{eqn:sheafheateqn} converge as $t\to\infty$ to the
  orthogonal projection of $x(0)$ onto $H^0(G;\Fc)$. 
\end{theorem}
\begin{proof}
  	The sheaf Laplacian $L_\Fc$ is symmetric and positive semidefinite, and hence 
	is diagonalizable with all eigenvalues nonnegative. The solution to 
	\eqref{eqn:sheafheateqn} is
  \[
  	x(t) = \exp(-t\alpha L_\Fc)x(0).
  \]	 
  	This solution operator has limit $\lim_{t\to\infty}\exp(-t\alpha L_\Fc)$ equal to zero
  	except on a block-diagonal identity submatrix corresponding to the zero eigenvalues.
	This is orthogonal projection onto $\ker L_\Fc = H^0(G;\Fc)$.
\end{proof}

Theorem~\ref{thm:sheafdiffusionsoln} has several consequences. One is that the only 
stable opinion distributions are global sections of the discourse sheaf. If this sheaf
has no nontrivial global sections, the only stable opinions will be everywhere zero: an 
uninteresting solution. Further, opinions converge exponentially to a
stable distribution, with rate of convergence related to the spectral properties
of the sheaf Laplacian \cite{hansen_toward_2019}.

The corresponding result for the graph Laplacian-based dynamics is that opinions
converge toward the average of the initial opinion distribution. Discrete-time
local averaging dynamics (with appropriately connected graphs) display a bit more
flexibility, but admit only a limiting distribution that is a weighted average of the
initial state. With sheaf Laplacian dynamics, new stable distributions are
possible.

\begin{ex} {\em (In polite company)}
  The simple sheaf shown in Figure~\ref{fig:polarizedsheaf} has all stalks of dimension
  one and all restriction maps of full rank: each person [vertex] has a private opinion
  about a topic and expresses that opinion. For the sake of illustration, assume that
  each stalk (vertex and edge) has the same basis topic --- 
  say, opinion about a certain politician. If this were a constant sheaf, a global section would
  represent consensus with identical opinions. However, as illustrated, the two agents on
  the left have a positive personal opinion, whereas the two on the right have a negative 
  personal opinion. The restriction maps encode expression of that opinion. Note how the
  discourse sheaf as illustrated permits selective expression of opinion. The two agents
  on the right tell a polite lie to their neighbors on the left but are frank with each other. 
  A global section of this sheaf maintains the public agreement. 

  One notes that although each agent {\em knows} the veracity of their opinion expressions,
  they do not know the veracity of their neighbors' expressed opinions. 
  Note also that, given this structure on the discourse sheaf, any
  initial opinion distribution will converge to one of these polarized
  distributions by Theorem~\ref{thm:sheafdiffusionsoln}. Compare the results
  of~\cite{altafini_dynamics_2012,altafini_consensus_2013}, which studied
  similar structures implicitly.
\end{ex}

% ********************************************************************************************
\begin{figure}
\label{fig:polarizedsheaf}
  \centering
  \includegraphics[width=5.0in]{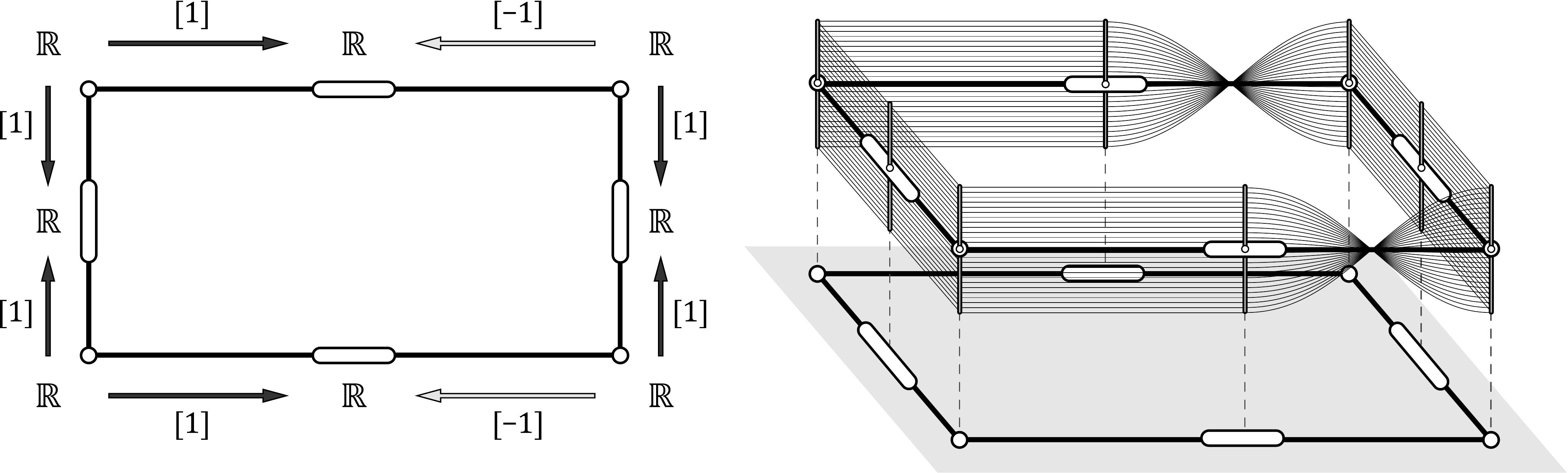}
  \caption{A sheaf supporting stable polarized opinion distributions as global sections. The two agents on the
  right lie to their neighbors on the left, but are truthful to each other.}
\end{figure}
% ********************************************************************************************

% ============================================================================================
\section{Stubbornness and Harmonic Extension}
\label{sec:HarmonicExtension}
% ============================================================================================
Consider a slight variation on the diffusion equation~\eqref{eqn:sheafheateqn} in which 
some agents are stubborn: they do not change their opinions in
response to communication with their neighbors. What consequence does this have
for the long-run dynamics? To answer this question, we consider the problem of
\introduce{harmonic extension} for partially-defined cochains on a sheaf. Our
results here and in the next two sections are extensions of ideas originally
introduced by Taylor~\cite{taylor_towards_1968} to the setting of discourse sheaves.

Let $U\subset V$ be a subset of vertices of $G$, and let $u\in C^0(U;\Fc)$ be a 0-cochain 
with support in $U$ --- a choice of a data $u_v \in \Fc(v)$ for each $v \in U$. A
harmonic extension of $u$ to the rest of the graph is a 0-cochain $x \in
C^0(G;\Fc)$ such that $x|_U = u$ and $(L_\Fc x)_v = 0 $ for every $v \in V
\setminus U$. Harmonic extensions always exist; when $H^0(G,U;\Fc) = 0$, they
are unique (see \cite[Proposition 4.1]{hansen_toward_2019} for a proof).

\begin{theorem}
	\label{thm:stubborndynamics}
  The $U$-restricted dynamics
  \begin{equation}
  \left.
  	\frac{dx}{dt}
  \right|_v 
  = 
  \left\{
  \begin{array}{cl}
  -\alpha (L_\Fc x)_v & \colon v\notin U \\
  0 & v\in U
  \end{array}
  \right.
  \end{equation}  
  on an initial condition $x_0$ converges exponentially to the 
  harmonic extension of $u=(x_0)|_U$ nearest to $x_0$.
\end{theorem}
\begin{proof}
  As $u=x|_U$ is constant, we can rewrite the dynamics as acting
  purely on $y=x|_Y$ for $Y = V\setminus U$. Using $L_\Fc[\cdot,\cdot]$ to
  denote the block submatrix restricted to the indicated vertex sets, we can
  express the dynamics on $y(t)$ as 
  \begin{equation}
    \frac{dy}{dt} = -\alpha(L_\Fc[Y,Y] y + L_{\Fc}[Y,U] u).
  \end{equation}
  The fixed points of this dynamical system are the 0-cochains $x$ where $L_\Fc
  x = 0$ on $Y$ and $u = (x_0)|_U$ --- the harmonic extensions of $u$. 
  As $L_\Fc[Y,Y]$ is a principal submatrix of a positive
  semidefinite matrix, it is positive semidefinite; if $H^0(G,Y;\Fc) = 0$, 
  it is positive definite. Further, $\im L_\Fc[Y,U] \subseteq \im
  L_\Fc[Y,Y] \perp \ker L_\Fc[Y,Y]$, so $\frac{dy}{dt}$ is always orthogonal to
  $\ker L_\Fc[Y,Y]$, and hence the dynamics preserve $\ker L_\Fc[Y,Y]$.
  Therefore, without loss of generality we can consider the system restricted
  to $\im L_\Fc[Y,Y]$. That is, write $y = y^\perp + y^\parallel$, where
  $y^\parallel \in \im L_\Fc[Y,Y]$ and $y^\perp \in \ker L_\Fc[Y,Y]$.  

  We now apply the general solution to an inhomogeneous linear ODE for
  $x^\parallel$ to obtain
  \begin{align*}
    y(t) &= e^{-t\alpha L_\Fc[Y,Y]}y_0^\parallel - \int_0^t e^{-(t-\tau)\alpha L_\Fc[Y,Y]} \alpha L_\Fc[Y,U]u d\tau \\
    &= e^{-t\alpha L_\Fc[Y,Y]} y_0^\parallel - \alpha^{-1}L_\Fc[Y,Y]^{\dagger}(I-e^{-\alpha L_\Fc[Y,Y]t})\alpha
    L_\Fc[Y,U]u \\
    &= e^{-t\alpha L_\Fc[Y,Y]} y_0^\parallel - L_\Fc[Y,Y]^{\dagger}(I - e^{-\alpha L_\Fc[Y,Y]t}) L_\Fc[Y,U] u.
  \end{align*}
  Here $L_\Fc[Y,Y]^\dagger$ is the Moore-Penrose pseudoinverse of $L_\Fc[Y,Y]$,
  which when restricted to $\im L_\Fc[Y,Y]$ is simply the inverse.
  As $t \to \infty$, this expression converges to
  \begin{equation}
    y^\parallel_\infty = -L_\Fc[Y,Y]^{\dagger} L_\Fc[Y,U] u.
  \end{equation}
  This is the minimum-norm harmonic extension of $u$, since any harmonic
  extension satisfies $L_\Fc[Y,Y]y + L_\Fc[Y,U]u = 0$. Since
  $y^\perp_0$ is unchanged throughout, the limit is therefore $y_\infty = y^\perp_0
  + y^\parallel_\infty$, which clearly still satisfies the equation for harmonic
  extension of $u$. 
\end{proof}

This result quantifies how even a few stubborn individuals can have an influence
on the opinion distribution throughout an entire social network. The opinion
distribution is kept in constant tension induced by the stubborn individuals. We
can characterize this equilibrium as the configuration minimizing total
disagreement given the stubborn agents' opinions. That is, the harmonic
extension is a minimizer of $\norm{\delta x}^2$ subject to $x|_U = u$.
If there is a global section $x$ with $x|_U = u$, it is clearly a harmonic extension
of $u$. However, harmonic extensions are not in general global sections.

% =========================================================================================
\section{Controlling Opinions}
\label{sec:Control}
% =========================================================================================
The analysis of stubborn individuals prompts the notion of more intentional control
of opinions over a social network. 
By way of merging the notation of the previous section with that of linear controls, let
$U\subset V$ be a set of {\em user input} vertices on which the controller can effect influence, and let $Y\subset V$
be a set of {\em observables} whose preferences one wants to control. We will denote by 
$u(t)\in C^0(U;\Fc)$ and $y(t)\in C^0(Y;\Fc)$ data supported on $U$ and $Y$ respectively
taking values in the stalks of $\Fc$ on those respective vertex sets.

The influence of user inputs on the system is mediated through a linear transformation 
$B:C^0(G;\Fc)\to C^0(G;\Fc)$ with image and coimage $C^0(U;\Fc)$; the observables are viewed through
$C:C^0(G;\Fc)\to C^0(G;\Fc)$ with image and coimage $C^0(Y;\Fc)$. The resulting linear control system is:
\begin{equation}
\label{eqn:sheafdiffusioncontrol}
  \frac{dx}{dt} = -\alpha L_{\Fc} x + B u 
  \quad : \quad
  \frac{dy}{dt} = Cx .   
\end{equation}

Controllability of \eqref{eqn:sheafdiffusioncontrol} answers the natural question
of whether opinion distributions on $Y$ can be determined via manipulation of
inputs on $U$. The system will naturally settle on a stable opinion distribution 
--- a global section of $\Fc$. Do inputs exist that will steer the system to an arbitrary
global section?

Consider the simplified case where $B$ is the identity map on  
$C^0(U;\Fc)$ and zero elsewhere. In this case, we have the following result:
\begin{theorem}
\label{thm:stabilizability}
  If the relative cohomology $H^0(G,U;\Fc)=0$, then the system~\eqref{eqn:sheafdiffusioncontrol}
  is stabilizable, with $B$ the identity on $C^0(U;\Fc)$.
\end{theorem}
\begin{proof}
  Stabilizability is equivalent to the condition that the matrix
  \[
  	\begin{bmatrix} 
  		(-\alpha L_\Fc - \lambda I) &B 
  	\end{bmatrix}
  \]
  be full rank for all $\lambda$ with nonnegative real part (see, e.g.,
  \cite{sontag_mathematical_1998} for a deeper discussion of this and other
  standard results of linear control theory). Since the
  eigenvalues of $-\alpha L_{\Fc}$ are real and nonpositive, $-\alpha L_\Fc -
  \lambda I$ is already full rank for any $\lambda$ with nonzero imaginary part
  or negative real part. Thus we only need to consider $\lambda = 0$, the case
  of the matrix $\begin{bmatrix}-\alpha L_\Fc &B\end{bmatrix}$. This matrix has
  full rank if for every $x \in \ker L_\Fc$, there exists some $u \in \im B$
  with $x^T u \neq 0$. In particular, this will be satisfied if no nontrivial global section of
  $\Fc$ vanishes on $U$. From the definition and interpretation of relative cohomology, 
  this is precisely the condition that $H^0(G,U;\Fc) = 0$. 
\end{proof}

This theorem implies that given an appropriate input set, we can ensure that the
dynamics converge to any global section of $\Fc$.
When $H^0(G,U;\Fc) = 0$, every global section is the unique harmonic
extension of its restriction to $U$, so we need only control
$u=x|_U$ to the desired states, and the rest of the network will
follow.

The dual result to Theorem~\ref{thm:stabilizability} (presented without proof) is
\begin{theorem}
\label{thm:controlability}
  If the relative cohomology $H^0(G,Y;\Fc)=0$, then the
  system~\eqref{eqn:sheafdiffusioncontrol} is detectable
  for $C$ the identity on $C^0(Y;\Fc)$ and zero elsewhere.
\end{theorem}
Thus, by observing agents on $Y$ we can educe any motion of the global state
projected to $H^0(G;\Fc)$. 
Given both conditions we can construct an observer-controller pair steering 
the system to a global section with any desired outcome as measured on $Y$.

\begin{ex}
  If $G$ is connected and the communication structure is given by the
  constant sheaf $\underline{\R^n}$, any vertex gives an input set for
  which the dynamics are stabilizable. This is because
  there are no nonzero constant $\R^n$-valued
  functions on $G$ that vanish at a vertex, so $H^0(G,\{v\};\underline{\R^n}) =
  0$. In terms of opinion dynamics, it is only necessary to have arbitrary
  influence on a single individual in order to ensure eventual global consensus
  on any given opinion --- a trivial property of the constant sheaf.
\end{ex}

% =========================================================================================
\section{Weighted Reluctance}
\label{sec:Reluctance}
% =========================================================================================

The control perspective is helpful in analyzing another type of linear dynamics,
where individuals are resistant to modifying their initial opinions (but not
infinitely so, as in Section~\ref{sec:HarmonicExtension}). We model this as a
feedback controller attached to the original system,
letting $u_v = \alpha\gamma_v((x_0)_v-x_v)$, where $\gamma_v$ is an
agent-dependent reluctance parameter. The long-run opinion distribution is again 
a function of harmonic extension. This can be effected by expanding the sheaf to 
an augmented graph $G'$ as follows (see Figure~\ref{fig:reluctance}). 

{\bf Construction:} Given $G$, augment it to a graph $G'$ by duplicating the
vertex set $V$ to a copy $V'$. Attach each $v'\in V'$ to the corresponding vertex
$v \in G$, with a single edge $e'$. Extend the discourse sheaf
$\Fc$ to a sheaf $\Fc'$ on $G'$ by letting $\Fc'(v') = \Fc'(e') = \Fc(v)$, where
$e'$ is the edge between $v'$ and $v$. The restriction maps are $\Fc'_{v \face e'}
= \Fc'_{v' \face e'} = \sqrt{\gamma_v} I$. One thinks of this augmented graph as
giving each agent an additional acquaintance, their {\em parent}, who acts as
a constant influence on their opinions.

% ********************************************************************************************
\begin{figure}
\label{fig:reluctance}
  \centering
  \includegraphics[width=4.0in]{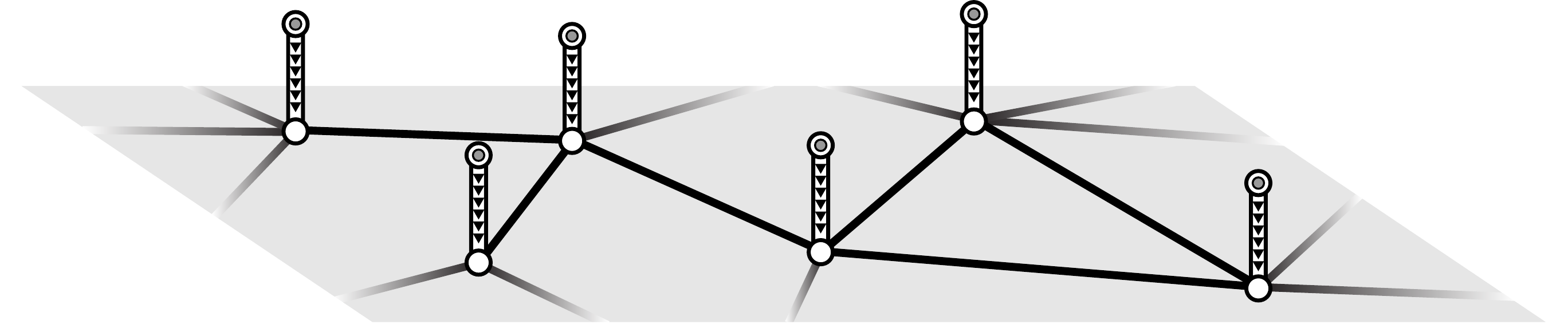}
  \caption{Weighted reluctance to changing opinions may be modeled by extending the 
  discourse sheaf over an augmented graph, giving each vertex a {\em stubborn parent}
  who exerts influence.}
\end{figure}
% ********************************************************************************************

We now apply the results about dynamics with stubborn agents to $\Fc'$ on $G'$, with 
parents as stubborn agents. 
For an initial condition $x_0\in C^0(G;\Fc)$, extend it to $G'$ via $x_0(v') = x_0(v)$ and
force all the agents $v'$ to be stubborn (in the sense of \S\ref{sec:HarmonicExtension}). 
The differential equation governing the opinion evolution is
\begin{equation}
\begin{split}
  \frac{dx_v}{dt} &= -\alpha(L_{\Fc'}x)_v \\
  &= -\alpha((L_{\Fc} x)_v + \gamma_v(x_v - x_{v'})) \\
  &= -\alpha(L_\Fc x)_v + \alpha\gamma_v((x_0)_v - x_v).
\end{split}
\end{equation}
Applying Theorem~\ref{thm:stubborndynamics}, we see that the opinion
distribution converges to the harmonic extension of $(x_0)|_{V'}$ to the rest of
$\Fc'$. The stubbornness parameters $\gamma_v$ influence how much the initial
opinions influence the limit, and hence how far from a global section of $\Fc$ the
limiting opinion distribution lies. 

% =========================================================================================
\section{Learning to Lie}
\label{sec:EvolvingExpression}
% =========================================================================================
The notion that communication over a social network leads to changes in opinions
is a convenient idealization to which diffusion dynamics applies. Because using a 
discourse sheaf to model opinion dynamics makes explicit the communication structure 
employed by the agents, it allow us to model changes in that structure. Instead of 
opinions changing over time, one could just as well consider evolution of expression:
agents can learn to communicate differently based on the reactions of their neighbors.
Leaving aside the sociological questions of whether a typical person 
in the face of opposition actually changes opinions or simply ``learns to communicate better,'' 
we demonstrate the flexibility of the discourse sheaf model under such settings.

Assume that each agent $v$ is able to modify all its restriction maps $\Fc_{v\face e}$
(for edges $e$ incident to $v$) and is able to observe their neighbors' translated
opinions $\Fc_{u \face e} x_u$. If the goal of each agent is to learn how to
translate their opinions so that apparent consensus is reached, they should
alter their restriction maps to remove the part of the image that contributes to
disagreement with neighbors. That is, the dynamics should be of the form
\begin{equation}\label{eqn:structuraldiffusion}
  \frac{d}{dt}{\Fc}_{v \face e} = -\beta (\Fc_{v \face e} x_v - \Fc_{u \face e}
  x_u)x_v^T ,
\end{equation}
for some diffusion strength $\beta>0$.
These dynamics may be nicely expressed in terms of the block rows $\delta_e$ of
the coboundary matrix corresponding to each edge:
\begin{equation}\label{eqn:edgediffusion}
  \frac{d\delta_e}{dt} = - \beta\delta_e x_e x_e^T,
\end{equation}
where $x_e$ is the vector in $C^0(G;\Fc)$ in which all entries corresponding to
vertices not incident to $e$ have been replaced with zero.
Combining these together, we have
\begin{equation}
  \frac{d\delta}{dt} = - \beta P_\delta(\delta x x^T),
\end{equation}
where $P_\delta$ is the map that takes the matrix of a linear transformation $C^0(G;\Fc) \to
C^1(G;\Fc)$ and projects it to a matrix with the correct sparsity pattern to be
a sheaf coboundary matrix. That is, $P_\delta$ sets all entries for blocks
corresponding to non-incident vertex-edge pairs to zero.

\begin{theorem}\label{thm:structuraldiffusion}
  For $\Fc(t)$ a solution to~\eqref{eqn:structuraldiffusion} on the space of sheaves over $G$ with fixed
  stalks, the sheaf $\Fc=\Fc(0)$ converges to $\Fc'=\lim_{t\to\infty}\Fc(t)$, the nearest 
  sheaf such that $x$ is a global section, where distance between $\Fc$ and $\Fc'$ is 
  measured by the squared Frobenius norm:
	\begin{equation}\label{eqn:sheafdistance}
	d(\Fc,\Fc') = \sum_{v \face e} \norm{\Fc_{v \face e}-\Fc'_{v\face e}}_F^2 = \norm{\delta - \delta'}_F^2.
	\end{equation}
\end{theorem}
\begin{proof}
  The $0$-cochain $x$ is a global section of $\Fc'$ precisely when, for each edge
  $e = u \sim v$, $\Fc'_{u \face e} x_u = \Fc'_{v \face e}x_v$, or equivalently,
  when $\delta_e'x_e = 0$.

  The dynamics for each edge are uncoupled, as can be seen from the form of the
  equation in~\eqref{eqn:edgediffusion}. The differential equation for each edge
  is linear, given by the operator $A$ taking $\delta_e$ to $\delta_e x_e
  x_e^T$. This operator is positive semidefinite with respect to the inner
  product $\ip{\delta_e,\delta_e'} = \tr(\delta_e^T\delta_e')$, and its kernel
  is given by those $\delta_e$ for which $\delta_e x_e = 0$. Therefore, the same
  argument as in the proof of Theorem~\ref{thm:sheafdiffusionsoln} shows that
  the trajectory of $\delta_e$ converges to its orthogonal projection onto $\ker
  A$ with respect to this inner product. The corresponding norm is the Frobenius
  norm, and hence $\delta_e$ converges to the nearest $\delta_e'$ as measured by
  this norm such that $\delta_e'x_e = 0$. 

  Since the distance~\eqref{eqn:sheafdistance} decomposes across edges, and each edge
  independently satisfies this distance-minimizing property, the same is true of
  their combination into a complete coboundary matrix. That is, the limiting
  coboundary matrix $\delta' = \lim_{t \to \infty} \delta(t)$ is the minimizer
  of $\norm{\delta(0) - \delta'}_F^2$ such that $\delta'$ is a sheaf coboundary
  matrix and $\delta'x = 0$. 
\end{proof}

This proof indicates a sort of duality between the sheaf heat
equation~\eqref{eqn:sheafheateqn} and the structural
dynamics~\eqref{eqn:structuraldiffusion}. Both are diffusion-like processes,
adjusting parameters to alleviate a local discrepancy. 

\begin{ex}
  {\em (Learning to lie)}
  The restriction map diffusion dynamics can convert the constant sheaf into a
  nontrivial communication structure. Start with the constant sheaf on a graph
  with two vertices and a single edge as shown in Figure~\ref{fig:learntolie},
  with a highly inconsistent 0-cochain assigning $-4$ to one vertex and $+1$ to
  the other. As the restriction map dynamics progress, the discourse sheaf becomes an
  inconsistent one, and the sign of one restriction map changes --- the
  corresponding agent \introduce{learns to lie} about their opinion. Note that
  this agent changes their expressed opinion's sign, but also downplays its magnitude,
  allowing the original opinion distribution to become a global section.
\end{ex}

% ********************************************************************************************
\begin{figure}
\label{fig:learntolie}
  \centering
  \includegraphics[width=5.0in]{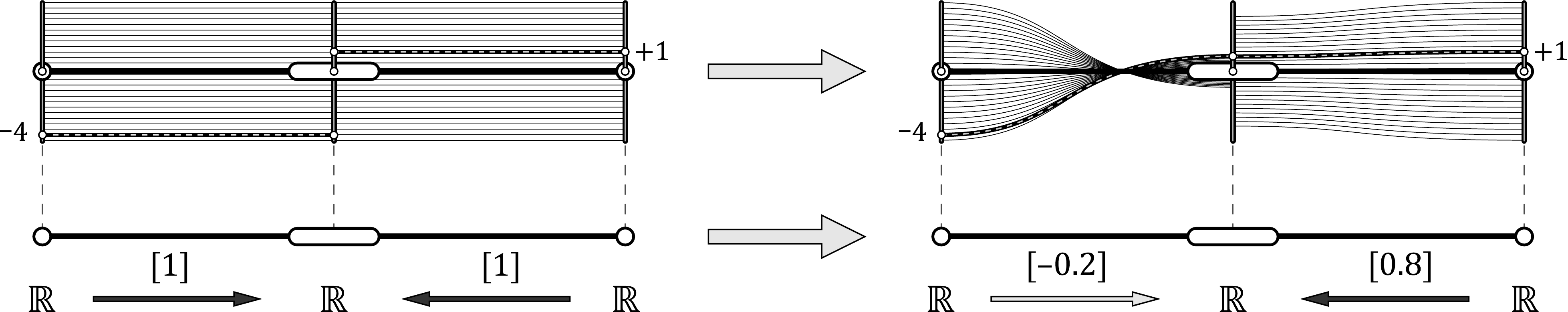}
  \caption{A constant sheaf over an edge with an initial opinion distribution that is highly 
  discordant [left] converges under diffusion of the sheaf to a nonconstant sheaf [right] in 
  which the agent with the more extreme negative opinion has ``learned to lie''
  in order to come to consensus.}
\end{figure}
% ********************************************************************************************

% =========================================================================================
\section{Joint Opinion-Expression Diffusion}
\label{sec:Joint}
% =========================================================================================
The natural culmination of this line of reasoning is to combine the opinion
diffusion~\eqref{eqn:sheafheateqn} and the communication structure
diffusion~\eqref{eqn:structuraldiffusion}. That is, evolve both the $0$-cochain
$x$ and the restriction maps $\Fc_{v \face e}$ according to
\begin{equation}
\label{eqn:combineddiffusion}
\begin{split}
  \frac{dx}{dt} &= -\alpha \delta^T\delta x \\
  \frac{d{\delta_e}}{dt} &= -\beta \delta_e x_e x_e^T.
\end{split}
\end{equation}

\begin{lemma}
For $x,\delta$ evolving according to~\eqref{eqn:combineddiffusion}, the function
	\begin{equation}
	\label{eq:lyapunov}
	  \Psi(x,\delta) = \frac{1}{2} x^T\delta^T\delta x
	\end{equation}	
satisfies $\Psi(x,\delta) \geq 0$ and $\frac{d}{dt}\Psi(x,\delta) \leq 0$, with zero
attained in both instances if and only if $\delta x = 0$. 
\end{lemma}
\begin{proof}
In \eqref{eqn:combineddiffusion}, $\delta$ is considered to lie in the space
of linear transformations with the appropriate sparsity pattern to be the coboundary of 
a sheaf over $G$. That $\Psi(x,\delta) \geq 0$ is immediate from its definition, as
is its vanishing precisely when $\delta x = 0$. The proof of the second inequality is by direct computation:
\begin{equation}
\begin{split}
  \frac{d}{dt}{\Psi}(x,\delta) &=  x^T\delta^T\delta \frac{dx}{dt} +  x^T\delta^T\frac{d\delta}{dt}x
  \\
  &= - \alpha x^T \delta^T\delta\delta^T\delta x - \beta x^T \delta^T
  P_\delta(\delta x x^T)x \\
  &= - \alpha x^T (\delta^T\delta)^2 x -  \beta \sum_e x_e^T \delta_e^T
  \delta_e x_e x_e^T x_e.
\end{split}
\end{equation}
The first term is clearly nonpositive, and negative precisely when $\delta x \neq 0$,
and the second term is similarly negative when $\delta_e x_e \neq 0$ for some $e$. 
\end{proof}

Applying LaSalle's invariance principle to this function reveals that all limit
points of trajectories of~\eqref{eqn:combineddiffusion} satisfy $\delta x = 0$. 
More is true: it is easy to check that~\eqref{eqn:combineddiffusion} is actually
a gradient descent equation on $\Psi$, with the gradient defined with respect to
the inner product $\ip{(x,\delta),(x',\delta')} =
\frac{1}{\alpha} x^T x' + \frac{1}{\beta} \sum_{e}\tr(\delta_e^T\delta_e')$.
This implies that trajectories indeed converge to points with $\delta x = 0$.
That is, trajectories converge to points describing a sheaf $\Fc$ on $G$ with
coboundary map $\delta$ together with a global section $x$ of $\Fc$.

Because the stationary points of~\eqref{eqn:combineddiffusion} are not isolated,
there is no global asymptotic stability. The set of equilibria is the set of solutions 
to a system of degree-2 polynomials forming a singular algebraic variety. The equilibria 
lying at nonsingular points of this variety are, however, Lyapunov stable.

\begin{theorem}
\label{thm:LyapunovStability}
  The smooth stationary points of~\eqref{eqn:combineddiffusion} are Lyapunov stable.
\end{theorem}
\begin{proof}
  It remains to show that given a stationary point $(x_*,\delta_*)$ at which the derivative of
  the map $(x,\delta) \mapsto \delta x$ is full rank, any neighborhood of $(x_*,\delta_*)$ 
  contains the forward orbit of a subneighborhood. 
  In a neighborhood of $(x_*,\delta_*)$, the set of equilibria
  of~\eqref{eqn:combineddiffusion} is a smooth manifold, given by the equation
  $\delta x = 0$. The tangent space at this point is the kernel of the map
  $(\delta,x) \mapsto \delta_* x + \delta x_*$. Meanwhile, the linearization
  of~\eqref{eqn:combineddiffusion} about $(x_*,\delta_*)$ is
  \begin{align*}
    \frac{dx}{dt} &= - \alpha(\delta_*^T\delta x_* + \delta_*^T\delta_* x) \\
    \frac{d\delta_e}{dt} &= -\beta\left[ \delta_e^T (x_*)_e (x_*)_e^T + (\delta_*)_e^T x_e (x_*)_e^T \right]
  \end{align*}
  These are zero if and only if $(x,\delta)$ satisfy $\delta_*x + \delta x_* = 0$,
  so the stationary set of the linearization is precisely the same as the
  tangent space at $(x_*,\delta_*)$. Therefore, the stationary points near
  $(x_*,\delta_*)$ describe a center manifold for the dynamics.
  Standard results on stability of the center manifold indicate that the stability of an
  equilibrium within the center-stable manifold is equivalent to its stability
  in the center manifold \cite{kelley_stability_1967}. Since $(x_*,\delta_*)$ is Lyapunov stable within the center
  manifold (because the dynamics are trivial), it is therefore Lyapunov stable.
\end{proof}

The regions of highest dimension of the variety $X$ are those where $\ker \delta = 0$ (for most
choices of $G$ and stalk dimensions). This would seem to indicate that we should
expect most trajectories to converge to points with $x = 0$. However, this is not the
case. 
\begin{theorem}\label{thm:conservedquantity}
  If one of the diagonal blocks of $\alpha \delta_0^T \delta_0 - \beta
  x_0^T x_0$ fails to be semidefinite, the trajectory of~\eqref{eqn:combineddiffusion}
  converges to a point $(x_\infty,\delta_\infty)$ with $x_\infty \neq 0$.
\end{theorem}
\begin{proof}
  Let $M = \delta^T\delta - xx^T$ and consider 
\begin{equation}
\begin{split}
  \frac{dM}{dt} &= \frac{d}{dt}
  \alpha\delta^T\delta - \beta x x^T\\
  &= \alpha[- \beta P_\delta(\delta x x^T)^T\delta - \beta\delta^TP_\delta(\delta x
  x^T) ]
  -\beta (-\alpha x x^T \delta^T \delta - \alpha\delta^T\delta x x^T)\\
  &= -\alpha\beta[(\delta^TP_\delta(x x^T\delta x x^T))^T + \delta^T P_\delta(\delta x x^T)]
  +\alpha \beta[(\delta^T\delta x x^T)^T + \delta^T\delta x x^T].
\end{split}
\end{equation}
Note that the diagonal block of $\delta^TP_\delta(\delta x x^T)$ corresponding to a
vertex $v$ is equal to $(\delta^T \delta x x^T)_{v,v}$, since the sparsity
pattern of the relevant block row of $\delta^T$ is the same as the sparsity
pattern imposed by the projection $P$. Thus, when we restrict to the diagonal
blocks, the derivatives cancel and we have $\frac{d}{dt} \text{diag}(M) = 0$.
Therefore, if $M_{v,v}$ is indefinite at $t = 0$, it must be indefinite for all $t$. 
In particular, this means that $xx^T$ cannot approach zero, since otherwise $M$ and
all of its diagonal subblocks would approach semidefiniteness.
\end{proof}

This condition implies that if any diagonal element of $M$ is negative, the
dynamics converge to a nonzero $x$. {\em A fortiori}, if $\tr(M) = \norm{\delta}_F^2 -
\norm{x}^2$ is negative, the limiting value of $x$ is nonzero. Thus, given any
initial $\delta_0$ and $x_0$, there exists some scaling factor $\kappa$ such that the
initial condition $(\delta_0, \kappa x_0)$ converges to a sheaf with a nontrivial global section.

Other relevant quantities are controlled during the combined diffusion dynamics.
\begin{theorem}
  The quantities $\norm{\delta}_F^2$, $\norm{x}^2$, $\norm{\delta x}^2$, and
  $\frac{\norm{\delta x}^2}{\norm{x}^2}$ are nonincreasing under the evolution of~\eqref{eqn:combineddiffusion}.
\end{theorem}
\begin{proof}
That $\norm{\delta x}^2$ is decreasing is implied by the fact that $\Psi$ is
decreasing under~\eqref{eqn:combineddiffusion}. We evaluate the other
derivatives:
\begin{equation}
\begin{split}
  \frac{d}{dt} \norm{\delta}_F^2 &=  \frac{d}{dt} \tr(\delta^T\delta) \\
  &= -\beta\tr(\delta^TP_\delta(\delta x x^T) + P_\delta(\delta x x^T)^T\delta) \\
  &= -2\beta \sum_e (x_e^T \delta_e^T \delta_e x_e) x_e^T x_e \leq 0,
\end{split}
\end{equation}
\begin{equation}
\begin{split}
  \frac{d}{dt} \norm{x}^2 &= \frac{d}{dt} x^Tx \\
  &=-\alpha(x^T\delta^T\delta x + (\delta^T \delta x)^T x) \\
  &= - 2\alpha x^T\delta^T\delta x \leq 0.
\end{split}
\end{equation}
Finally,
\begin{equation}
  \frac{d}{dt} \frac{\norm{\delta x}^2}{\norm{x}^2} = \frac{ \norm{x}^2
  \frac{d}{dt}\norm{\delta x}^2 - \norm{\delta x}^2 \frac{d}{dt}
  \norm{x}^2}{\norm{x}^4}. 
\end{equation}
We know that $\frac{d}{dt} \norm{\delta x}^2 \leq -\alpha x^T
(\delta^T\delta)^2 x$, so this is bounded above by
\[
   \frac{-2\alpha \norm{x}^2 (\alpha x^T(\delta^T\delta)^2 x)
     + 2\alpha(x^T \delta^T \delta x)^2}{\norm{x}^4} 
   = 2\alpha \left[ \left( \frac{x^T\delta^T\delta x}{\norm{x}^2}  \right)^2
    -\frac{x^T(\delta^T\delta)^2 x}{\norm{x}^2}\right].
\]
Thus, $\frac{\norm{\delta x}^2}{\norm{x}^2}$ is decreasing if the inequality
\[ 
	\left( \frac{x^T\delta^T\delta x}{\norm{x}^2}\right)^2 
	\leq
  	\frac{x^T(\delta^T\delta)^2 x}{\norm{x}^2}
\]
holds. By taking $\norm x = 1$ and diagonalizing $\delta^T\delta$, we get the
equivalent inequality 
\[
	\sum_{i} \lambda_i^2 x_i \geq \left( \sum_i \lambda_i x_i \right)^2
\]
for $\lambda_i, x_i \geq 0$, $\sum_i x_i = 1$, which is simply Jensen's inequality.
\end{proof}

This last decreasing observable is the Rayleigh quotient of $L = \delta^T
\delta$ corresponding to $x$. That it is decreasing means $x(t)$ becomes a global section
of $\Fc(t)$ at least as quickly as it approaches zero. 

\begin{ex}
  {\em (Learning to lie, redux)} 
  The combined diffusion dynamics also enable an agent to learn to falsify
  opinions. Consider the same initial conditions as before, but run the combined
  dynamics with $\alpha = \beta = 1$. The sheaf and cochain converge again to a
  sheaf where one agent lies about their opinion. The opinion distribution
  changes somewhat, but because the discrepancy between opinions is so great,
  the sheaf is able to adapt more easily than the opinions themselves: 
  see Figure~\ref{fig:learntolie2}.
\end{ex}

% ********************************************************************************************
\begin{figure}
\label{fig:learntolie2}
  \centering
  \includegraphics[width=5.0in]{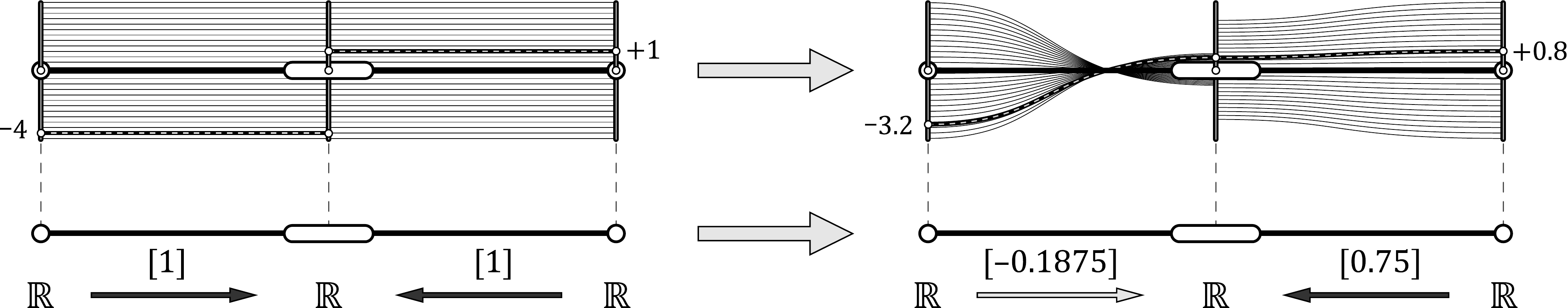}
  \caption{A constant sheaf over an edge with an initial opinion distribution that is highly 
  discordant [left] converges under the combined opinion-expression diffusion to a nonconstant 
  sheaf [right] in which both opinions and expressions have relaxed to come to consensus. Under
  this initial condition, the agent on the left has learned to lie.}
\end{figure}
% ********************************************************************************************

% =========================================================================================
\section{Nonlinear Laplacians}
\label{sec:NonlinearLaplacians}
% =========================================================================================
The sheaf Laplacian is constructed from the sheaf coboundary and its adjoint,
with an implicit isomorphism between $C^1$ and its dual given by the standard
inner product on the edge stalks. We can make this mapping explicit, and insert
a new function between the terms in order to produce nonlinear Laplacians with
new behaviors. 

Let $\phi_e: \Fc(e) \to \Fc(e)$ be a continuous but not-necessarily-linear map
for each edge $e$ of $G$, and define $\Phi: C^1(G;\Fc) \to C^1(G;\Fc)$ by
combining these edgewise maps. The corresponding nonlinear Laplacian is
$L_\Fc^\Phi = \delta_\Fc^T \circ \Phi \circ \delta_\Fc$. Because the nonlinear
map $\Phi$ is applied edgewise, $L_\Fc^\Phi x$ can still be computed locally in
the network. Thus, the nonlinear heat equation
\begin{equation}\label{eqn:nonlinearheat}
  \frac{dx}{dt} = -\alpha L_\Fc^\Phi x
\end{equation}
over a discourse sheaf $\Fc$ describes a form of
networked opinion dynamics.

One way to construct a nonlinear Laplacian is by beginning with a set of edge
potential functions. The heat equation on a sheaf is gradient descent with
respect to $x$ on the potential function $\Psi(x) = \frac{1}{2}\norm{\delta x}^2 =
\sum_{e} \frac{1}{2}\norm{\delta_e x}^2$. We can replace this potential function
with a new function defined edgewise:
\[\Psi(x) = \sum_{e} U_e(\delta_e x) = U(\delta x),\]
for once-differentiable edgewise potential functions $U_e: \Fc(e) \to \R$.
The gradient of this function at a point $x$ is $\nabla \Psi(x) = \delta^T \circ
\nabla U \circ \delta x$. This is therefore a nonlinear sheaf Laplacian
$L_\Fc^\Phi$ with $\Phi = \nabla U$. The heat equation for this nonlinear
Laplacian is precisely gradient descent on $\Psi$.  

Analysis of the heat equation is simplified for nonlinear Laplacians constructed
from edge potentials. For instance, if the potential functions $U_e$ are convex,
$\Psi$ serves as a Lyapunov function, ensuring stability of the dynamics. If each
$U_e$ has a local minimum at $0$, global sections of $\Fc$ will be stationary
points of the heat equation. Indeed, if each $U_e$ is radially unbounded with a
single local minimum at $0$, the long-term behaviors of the linear and nonlinear heat
equations agree.

\begin{proposition}
  For each edge $e$ of $G$, let $U_e: \Fc(e) \to \R$ be a differentiable,
  radially unbounded
function with a unique local minimum at $0$. Trajectories of the nonlinear heat
equation
  \[\frac{dx}{dt} = -\alpha L_\Fc^{\nabla U} x\]
  converge to the orthogonal projection of the initial condition onto $H^0(G;\Fc)$.
\end{proposition}
\begin{proof}
  First observe that $\frac{dx}{dt} \in \im \delta^T$ and hence is always
  orthogonal to $\ker \delta = H^0(G;\Fc)$. Decomposing $x = x^\parallel +
  x^\perp$, where $x^\parallel$ is the orthogonal projection of $x$ onto
  $H^0(G;\Fc)$, we restrict our attention to $\im \delta^T$ and the evolution of $x^\perp$. 

  Letting $\Psi(x^\perp) = U(\delta x^\perp)$, we have a function
  vanishing precisely when $\delta x^\perp = 0$, which holds precisely when
  $x^\perp = 0$. Further, 
\[\frac{d\Psi}{dt} = \ip{\nabla \Psi(x^\perp),\frac{d
      x^\perp}{dt}} = \ip{L_\Fc^{\nabla U} x^\perp,-\alpha L_{\Fc}^{\nabla
      U}x^\perp} \leq 0,\]
with equality precisely when $L_\Fc^{\nabla U} x^\perp = 0$. Because $U$ 
is radially unbounded and has a unique local minimum at $0$, $\nabla U$ vanishes
only at $0$, and hence $L_\Fc^{\nabla U} x^\perp = 0$ if and only if $x^\perp =
0$. Thus $\frac{d\Psi}{dt}$ vanishes only at the origin and is therefore a Lyapunov
function for the nonlinear heat equation restricted to $\im \delta^T$. Radial
unboundedness of $U$ implies radial unboundedness of $\Psi$ and therefore the
origin is globally asymptotically stable, meaning $x^\perp \to 0$. Therefore,
$\lim_{t \to \infty} x(t) = x^\parallel(0)$.  
\end{proof}

% =========================================================================================
\section{Bounded Confidence}\label{sec:BoundedConfidence}
% =========================================================================================
Nonlinear Laplacian dynamics (and in particular edge potential dynamics) can be
used to extend the popular {\em bounded confidence} opinion dynamics to discourse
sheaves. The central idea behind bounded confidence opinion dynamics is that
individuals only have confidence in the opinions of neighbors that are
sufficiently similar to their own, and thus only take these opinions into
account when updating. The reigning discrete-time model of bounded confidence
dynamics is based on that of Hegselmann and Krause
\cite{hegselmann_opinion_2002}, with extensions to multidimensional opinions. In
this model, each agent has a threshold $D$, and only pays heed to opinions of
neighbors that are within distance $D$ of their own opinion. Continuous-time
versions of this model have been analyzed, both with sharp discontinuous
thresholds \cite{blondel_continuous-time_2010,ceragioli_continuous-time_2011} as
well as smooth transitions between confidence and no-confidence
\cite{ceragioli_continuous_2012}.

We will here discuss a continuous-time version of bounded confidence dynamics
for sheaves, using smooth threshold functions. These will be represented in
terms of edgewise potential functions. Given a discourse sheaf $\Fc$ on a graph
$G$, for each edge $e$ of $G$ choose a threshold $D_e$ and a differentiable function $\psi_e:
[0,\infty) \to \R$ such that $\psi_e'(y) = 0$ for $y \geq D_e$ and $\psi_e'(y) >
0$ for $y < D_e$. The edge
potential function $U_e: \Fc(e) \to \R$ is then given by $U_e(y_e) =
\psi_e(\norm{y_e}^2)$. The gradient of this potential is $\nabla U_e(y_e) =
\psi_e'(\norm{y_e}^2)y_e$, and therefore the associated nonlinear Laplacian is
given by 
\begin{equation}
	L_\Fc^{\nabla U}x = \delta^T \diag(\psi_e'(\norm{\delta_e x}^2)) \delta x.
\end{equation}
This formula can be written vertexwise as
\begin{equation}
	(L_\Fc^{\nabla U}x)_v = \sum_{u,v \face e} \Fc_{v \face
    e}^T\psi_e'(\norm{\Fc_{v \face e} x_v - \Fc_{u \face e}x_u}^2)(\Fc_{v \face
    e}x_v - \Fc_{u \face e} x_u).
\end{equation}
In comparison with the formula for the standard sheaf Laplacian, there is a
nonlinear scaling factor depending on the discrepancy over each edge. This nonlinear
Laplacian $L_\Fc^{\nabla U}$ can be used to generate bounded confidence dynamics.
\begin{theorem}
	\label{thm:nonlinearlaplaciankernel}
  	Suppose that $U_e(y_e) = \psi_e(\norm{y_e}^2)$ for some $\psi_e: [0,\infty)
  	\to \R$ with $\psi_e'(y) = 0$ for $y \geq D_e$ and $\psi_e'(y) > 0$ for $y < D_e$. 
	Then an opinion distribution $x \in C^0(G;\Fc)$ is harmonic with respect to $L_\Fc^{\nabla U}$
	if and only if 
	for every edge $e = u \sim v$ with $\Fc_{v \face e} x_v \neq
	\Fc_{u \face e} x_u$, $\norm{\Fc_{v \face e} x_v - \Fc_{u \face e} x_u}^2 \geq D_e$. 
\end{theorem}
\begin{proof}
  If $x \in H^0(G;\Fc)$, then $\delta x = 0$, and since $\nabla U_e(0) = 0$,
  $\delta^T \nabla U (\delta x) = \delta^T 0 = 0$. More generally, $\nabla U_e(\delta x)_e
  = 0$ whenever either $(\delta_x)_e = 0$ or $\norm{(\delta x)_e}^2 \geq D_e$. But
  $(\delta x)_e = \Fc_{v \face e} x_v - \Fc_{u \face e} x_u$.

  Conversely, if $L_\Fc^{\nabla U} (x) = 0$, then $\nabla U(\delta x) \in \ker \delta^T$;
  equivalently, $\nabla U(\delta x)$ is orthogonal to $\im \delta$. In particular,
  $\ip{\nabla U(\delta x),\delta x}$ must be zero. Letting $y = \delta x$, we have
  \begin{equation}
	  \ip{\nabla U(y),y}
	  = \sum_e \ip{\nabla U_e(y_e),y_e} = \sum_e
	  \ip{\psi_e'(\norm{y_e}^2)y_e,y_e}.
  \end{equation}
These terms are all nonnegative, and $\ip{\psi_e'(\norm{y_e}^2)y_e,y_e} = 0$ if and
only if either $y_e = 0$ or $\psi_e'(\norm{y_e}^2) = 0$. The first condition holds
precisely when $(\delta x)_e = 0$, and the second holds precisely when
$\norm{(\delta x)_e}^2 \geq D_e$. 
\end{proof}

Naturally, one constructs the bounded confidence Laplacian to study the
corresponding diffusion dynamics
\begin{equation}\label{eqn:bcdynamics}
  \frac{dx}{dt} = - L_\Fc^{\nabla U}(x).
\end{equation}
Theorem~\ref{thm:nonlinearlaplaciankernel} identifies the equilibria of these
dynamics. Global sections of $\Fc$ are still equilibria, but there are more. Given $x
\in C^0(G;\Fc)$, we construct a subgraph $G_x$ of $G$ and a sheaf $\Fc_x$ on
$G_x$ as follows: % (see Figure~\ref{fig:boundedconfidence}): 
$G_x$ has the same vertices as $G$, but only contains the edges $e$ where
$\norm{(\delta x)_e} < D_e$. The sheaf $\Fc_x$ is the same as $\Fc$ but with the
data for edges not in $G_x$ removed. So a 0-cochain $x \in C^0(G;\Fc)$ is a
fixed point for~\eqref{eqn:bcdynamics} if and only if it is a global section of
$\Fc_x$. We denote the subset of $C^0(G;\Fc)$ for which the corresponding
sheaf is $\Fc_x$ by $K_{x}$. That is, 
\begin{equation}
	K_x = \{x \in C^0(G;\Fc) : \norm{\delta_e x}^2 < D_e 
		\text{ for } 
	e \in G_x, \norm{\delta_e x}^2 \geq D_e \text{ for } e \notin G_x\}.
\end{equation}
Suppose that $x_0$ is a fixed point with $\norm{(\delta x_0)_e}^2 \neq D_e$ for all edges
$e$. Sufficiently close to $x_0$, the dynamics behave like the standard
diffusion dynamics on $\Fc_{x_0}$.

\begin{theorem}\label{thm:bcstability}
  Let $x_0$ be a fixed point of~\eqref{eqn:bcdynamics} lying in the interior of
  $K_{x_0}$---that is, with $\norm{(\delta
    x_0)_e} > D_e$ for all $e \notin G_{x_0}$. There exists a neighborhood $W$
  of $x_0$ such that for every initial condition $x \in W$, the trajectory
  of~\eqref{eqn:bcdynamics} converges to the nearest global section of $\Fc_{x_0}$.
\end{theorem}
\begin{proof}
  Let $W$ be a neighborhood of $x_0$ contained in $K_{x_0}$ satisfying:
  \begin{enumerate}
  \item
    If $x \in W$, its orthogonal projection $x^\parallel$ onto
    $H^0(G_{x_0};\Fc_{x_0})$ is in $W$.
  \item
    If $x \in W$ is not a fixed point with $x = x^\parallel + x^\perp$, 
     then $\norm{x^\perp} < \frac{\norm{\delta_e
        x^\parallel}^2-D_e}{2\norm{\delta_e x^\parallel}\norm{\delta_e}}$
    for all $e \notin G_{x_0}$.
  \item
    If $x = x^\parallel + x^\perp \in W$ is not a fixed point, then
    $\norm{x^\perp}^2 < \frac{D_e}{\norm{\delta_e}^2}$ for all $e \in G_{x_0}$. 
  \end{enumerate}
  Such a neighborhood exists by continuity of the linear maps $\delta_e$
  and because we can choose arbitrarily small tubular neighborhoods around the
  set of fixed points.

  Consider $x = x^\parallel + x^\perp \in W$. We will show that $x$ converges
  to $x^\parallel$, which is in $U$ by condition (1). Note that as long as $x \in K_{x_0}$,
  $\frac{d}{dt}x^\parallel = 0$. Similarly, for $x \in K_{x_0}$,
  $\frac{d}{dt}\frac{1}{2} \norm{x^\perp}^2 = \ip{x^\perp,\frac{d}{dt}x^\perp} =
  \ip{x^\perp,-\delta^T \nabla U(\delta x^\perp)} = - \ip{\delta x^\perp,\nabla U(\delta
    x^\perp)} \leq 0$, by the argument in the proof of
  Theorem~\ref{thm:nonlinearlaplaciankernel}. This is zero only if $x^\perp =
  0$, and hence $\norm{x^\perp}^2$ is strictly decreasing as long as it is
  nonzero and $x \in K_{x_0}$. In particular, for any initial condition $x \in
  W$ with $x^\perp \neq 0$ there is some maximal time interval
  $[0,T)$ on which $\frac{d}{dt}\norm{x^\perp}< 0$ is strictly decreasing.

  Combining this with (2) and (3) above reveals that on this time
  interval, for every $e \notin G_{x_0}$,
\begin{align*}
  \norm{\delta_e x }^2 &\geq \norm{\delta_e x^\parallel}^2 - 2\norm{\delta_e
    x^\perp}\norm{\delta_e x^\parallel} \\
  & \geq \norm{\delta_e x^\parallel}^2 - 2\norm{\delta_e}
    \norm{x^\perp}\norm{\delta_e x^\parallel} \\
  &>  \norm{(\delta x^\parallel)_e}^2 - 2\norm{\delta_e}
    \norm{\delta_e x^\parallel}\frac{\norm{\delta_e
        x^\parallel}^2-D_e}{2\norm{\delta_e x^\parallel}\norm{\delta_e}} = D_e
\end{align*}
by condition (2).
Similarly, for $e \in G_{x_0}$,
\begin{equation*}
  \norm{(\delta x)_e}^2  =  \norm{(\delta x^\perp)_e}^2
  \leq \norm{\delta_e}^2\norm{x^\perp}^2 
  < D_e
\end{equation*}
by condition (3). Since these relations hold at time $t = 0$ and
$\norm{x^\perp}$ is decreasing, they hold for all $t \in [0,T)$ as well.
This ensures that $x$ remains in $K_{x_0}$ and does not approach the boundary on
this interval. If $T$ is finite, we may thus conclude that $\frac{d}{dt}\norm{x^\perp} < 0$ at $t
= T$ as well, so it must be that that $T = \infty$. Thus $x$ remains in $K_{x_0}$ and
$\norm{x^\perp}$ is strictly decreasing for all time.   

It remains to show that trajectories converge: that $x^\perp \to 0$. This
happens because $\Psi_{x_0}(x^\perp) = \sum_{e \in G_{x_0}}U_e(\delta_e x^\perp)$
is strictly decreasing
along trajectories and vanishes precisely when $x^\perp = 0$.
\end{proof}

% =========================================================================================
\section{Antagonistic Dynamics}
\label{sec:AntagonisticDynamics}
% =========================================================================================
One may wish to model certain agent pairs who, rather than attempting to
move toward mutual consensus, instead try to increase the distance
between their expressed opinions. Such a dynamical structure might correspond to
relationships of distrust or enmity. We will call such a relationship a 
\introduce{negative edge} and partition the edge set of the graph $G$ into 
negative edges, $E_-$, and complementary positive edges $E_+$.

One way to model relationships over negative edges in a discourse sheaf is to change the sign of one
restriction map on each negative edge. Instead of a tactful lie, this
might be interpreted as an ``agreement to disagree'' between the neighbors,
making a stable pair of opinions one where the expressed opinions are the same
in magnitude but opposite in sign. In the case where the discourse sheaf is
simply a constant sheaf, this is essentially the approach taken 
in~\cite{altafini_dynamics_2012,altafini_consensus_2013}. 

A better approach is to leave the discourse sheaf untouched and modify the
dynamics appropriately. Let $E_-$ be the set of negative links between agents,
and $E_+$ its complement, the set of positive links. We define edge potential
functions associated to this signing:
\[
	U_e(y) =
		\begin{cases} 
			\norm{y}^2 & e \in E_+ \\
		    -\norm{y}^2 & e \in E_-
	  	\end{cases}.
\]
The corresponding nonlinear Laplacian $L_\Fc^{\nabla U}$ describes dynamics where agents attempt to
move toward consensus as quickly as possible over positive edges and away from
consensus as quickly as possible over negative edges. Let $S: C^1(G;\Fc) \to
C^1(G;\Fc)$ be the block diagonal matrix whose blocks corresponding to positive
edges are $I$ and whose blocks corresponding to negative edges are $-I$. The
nonlinear Laplacian for this edge potential is the matrix
$L_\Fc^S = \delta^T S \delta$, so it is in fact a linear operator.\footnote{As is common
  in mathematics, ``nonlinear'' here really means ``not-necessarily-linear.''}
These dynamics are more akin to those considered in~\cite{altafini_predictable_2015}.

The kernel of $L_\Fc^S$ contains $H^0(G;\Fc)$, but may be larger.  
Further, this signed sheaf Laplacian is not necessarily positive semidefinite, so its
corresponding diffusion dynamics may be unstable. One simple situation in which
this happens is as follows:
\begin{proposition}
  Suppose $E_-$ is a cutset of $G$, dividing the graph into subgraphs $G_0$ and
  $G_1$. If the natural map $H^0(G,G_1;\Fc) \to H^0(G_0;\Fc)$ is not
  surjective --- that is, if there exists a local section on $G_0$ which does not
  extend by zero to a global section of $G$ --- then $L_\Fc^S$ is indefinite.
\end{proposition}
\begin{proof}
  Take some nonzero $x \in H^0(G_0;\Fc)$ which is not in the image of this map.
  Extending $x$ by zero to the rest of $G$, we have $\delta_e x \neq 0$ for some
  $e \in E_-$, but $\delta_e x = 0$ for all $e \in E_+$. Thus, $\ip{x,L_\Fc^Sx}
  = \sum_{e \in E_-} \ip{\delta_e x,-\delta_e x}$. This is negative because
  there is at least one nonzero term.
\end{proof}

For the case of the constant sheaf on $G$, the relevant map is never surjective,
so a graph with a cutset of negative edges always has an indefinite signed
Laplacian, and hence unstable opinion dynamics. 

% =========================================================================================
\section{Conclusions}
\label{sec:Conc}
% =========================================================================================
This work introduced a novel and compelling ensemble of techniques
from cellular sheaves, sheaf cohomology, and sheaf Laplacians, to model and
analyze opinion dynamics over networks. The increase in mathematical sophistication
comes with an increase in expressiveness of the model: private opinions, personalized 
expressions of opinions, evolution of communication structures, stubbornness, 
obfuscation, and bounded confidence are all easily expressed using discourse sheaves
and sheaf diffusion dynamics. Despite this, there is no increase in difficulty of 
computation or analysis. The language of harmonic extension converts subtle 
solvability conditions to simple linear-algebraic cohomology computations. 
Despite the imposing terminology, sheaf cohomology is a concise and 
computable tool for determining feasibility of solving problems of existence, extension, 
approximation, controllability, and observability on sheaf dynamics.  

The diffusive dynamics studied here have been linear or near-linear.
Deeper analysis of the nonlinear dynamical systems \eqref{eqn:combineddiffusion} and
\eqref{eqn:bcdynamics} is needed, as well as the many other possible extensions.
These include incorporation of probabilistic elements or including 
multi-agent interactions using sheaves on simplicial complexes. 

Since they are defined in abstract structural terms, sheaf dynamics are applicable to more than simply 
opinion dynamics. This paper may be regarded as an introduction to the dynamics of cellular sheaves 
as a broad generalization of network dynamics, with social networks and opinion dynamics serving 
as a running example. Other examples of networked dynamics effected by local evolution
operators on rich data structures may be found in, e.g., neuroscience and game theory, at least.
For the sake of accessibility, much of the language and techniques of sheaf theory and sheaf cohomology
have been excised from this paper. A full incorporation of sheaf-theoretic operations would increase
the precision, concision, and generality of this work. 

%\bibliographystyle{siamplain}

%\bibliography{sheafspectra}

\end{document}